\documentclass[a4paper,11pt]{amsart}
\usepackage{amssymb,amscd}
\usepackage[cp1251]{inputenc}
\usepackage[T2A]{fontenc}
\usepackage{amsmath}
\usepackage{amsthm}
\usepackage{mathtext}
\usepackage{euscript}

\newtheorem{theorem}{Theorem}[section]
\newtheorem{proposition}[theorem]{Proposition}
\newtheorem{lemma}[theorem]{Lemma}

\newtheorem{question}[theorem]{Question}

\theoremstyle{definition}

\newtheorem{example}[theorem]{Example}
\newtheorem{construction}[theorem]{Construction}

\theoremstyle{remark}

\numberwithin{equation}{section}

\newcommand{\field}[1]{\mathbb{#1}}
\def\C{\field{C}}

\newcommand{\R}{\field{R}}

\newcommand{\Z}{\field{Z}}

\def\le{\leqslant}
\def\ge{\geqslant}

\DeclareMathOperator{\ord}{ord}

\begin{document}

\title{Calabi--Yau hypersurfaces and SU-bordism}

\author{Ivan Limonchenko}
\address{School of Mathematical Sciences, Fudan University, 220 Handan Road, Shanghai, 200433, P.R. China}
\email{ilimonchenko@fudan.edu.cn}

\author{Zhi L\"u}
\address{School of Mathematical Sciences, Fudan University, 220 Handan Road, Shanghai, 200433, P.R. China}
\email{zlu@fudan.edu.cn}

\author{Taras Panov}
\address{Department of Mathematics and Mechanics, Moscow
State University, Leninskie Gory, 119991 Moscow, Russia}
\address{Institute for Information Transmission Problems, Russian Academy of Sciences, Moscow}
\address{Institute of Theoretical and Experimental Physics, Moscow}
\email{tpanov@mech.math.msu.su}

\thanks{The first author was supported by the General Financial Grant from the China Postdoctoral Science Foundation, grant no. 2016M601486. The second author 
was supported by NSFC, grants no.~11371093,  11661131004 and 11431009. The third author was supported by the Russian Foundation for Basic Research, grants no.~17-01-00671 and 16-51-55017, and by Simons-IUM fellowship.}



\begin{abstract}
%In~\cite{Ba} 
Batyrev constructed a family of Calabi--Yau hypersurfaces dual to the first Chern class in toric Fano varieties. Using this construction, we introduce a family of Calabi--Yau manifolds whose SU-bordism classes generate the special unitary bordism ring $\varOmega^{SU}[\frac{1}{2}]\cong\mathbb{Z}[\frac{1}{2}][y_{i}\colon{i\ge 2}]$. We also describe explicit Calabi--Yau representatives for multiplicative generators of the $SU$-bordism ring in low dimensions.
\end{abstract}

\maketitle

%\section{Introduction}

\section{Preliminaries: $SU$-bordism and toric geometry}
A \emph{stably complex structure} (a \emph{unitary structure}, or a \emph{$U$-structure}) on a smooth manifold $M$ is determined by a choice of an isomorphism between 
the stable tangent bundle of $M$ and a complex vector bundle~$\xi$:
\begin{equation}\label{stabcom}
  c_{\mathcal T}\colon\mathcal T{M}\oplus\underline{\mathbb{R}}^{N}\stackrel\cong\longrightarrow\xi.
\end{equation}
Equivalently, a stably complex structure is the homotopy class of a lift of the map $M\to BO$ classifying the tangent bundle $\mathcal TM$ to a map $M\to BU$. A \emph{stably complex manifold} is a pair $(M,c_{\mathcal T})$.

A \emph{special unitary structure} (an \emph{SU-structure}) on $M$ is a stably complex structure $c_{\mathcal T}$ with a choice of an $SU$-structure on the complex vector bundle~$\xi$. Equivalently, an $SU$-structure is the homotopy class of a lift of the map $M\to BU$ classifying $\xi$ to a map $M\to BSU$. A stably complex manifold $(M,c_{\mathcal T})$ admits an $SU$-structure if and only if 
the first (integral) Chern class of $\xi$ vanishes: $c_1(\xi)=0$. Furthermore, such an $SU$-structure is unique if $H^1(M;\mathbb Z)=0$. 

We refer to a compact K\"ahler manifold $M$ with $c_1(M)=0$ as a \emph{Calabi--Yau manifold}. (Apparently, this is the most standard definition; however, other definitions of a Calabi--Yau manifold, sometimes non-equivalent to this one, also appear in the literature.) According to the theorem of Yau, conjectured by Calabi, a Calabi--Yau manifold admits a K\"ahler metric with zero Ricci curvature (for this, only vanishing of the  first \emph{real} Chern class is required). By definition, a Calabi--Yau manifold is an $SU$-manifold.

We have 
\[  
  H^*(BU(n);\Z)\cong\Z[c_1,\ldots,c_n],\quad \deg c_i=2i,
\] 
where the $c_i$ are the universal Chern characteristic classes. Given a
sequence $\omega=(i_1,\ldots,i_n)$ of nonnegative integers, define the monomial $c_\omega=c_1^{i_1}\cdots c_n^{i_n}$ of degree
$2\|\omega\|=2\sum_{k=1}^n k\,i_k$ and the corresponding
characteristic class $c_\omega(\xi)$ of a complex $n$-plane
bundle~$\xi$. The corresponding tangential Chern
\emph{characteristic number} of a stably complex manifold $M$ is
defined by 
\[
  c_\omega(M):=\langle c_\omega(\mathcal T M),[M]\rangle.
\]
Here $[M]$ is the fundamental homology class of~$M$,
and $\mathcal T M$ is regarded as a complex bundle via the
isomorphism~\eqref{stabcom}. The number $c_\omega[M]$ is assumed
to be zero when $2\|\omega\|\ne\dim M$.

%\begin{theorem}
%Two stably complex manifold $M$ and $N$ represent the same bordism
%classes in~$\varOmega^U$ if and only if their sets of Chern
%characteristic numbers coincide.
%\end{theorem}

One important characteristic class is $s_n$. It is defined as
the polynomial in $c_1,\ldots,c_n$ obtained by expressing the
symmetric polynomial $x_1^n+\cdots+x_n^n$ via the elementary
symmetric functions $\sigma_i(x_1,\ldots,x_n)$ and then replacing
each $\sigma_i$ by~$c_i$. Define the corresponding characteristic
number as 
\[
  s_n(M):=\langle s_n(\mathcal T M),[M]\rangle.
\]
It is known as the \emph{$s$-number} or the \emph{Milnor number} of~$M$.

For any integer $i\ge 1$, set 
\[
  m_{i}:=\begin{cases}
  1&\text{if $i+1\neq p^s$ for any prime $p$;}\\
  p&\text{if $i+1=p^s$ for some prime $p$ and integer $s>0$.}
  \end{cases}
\]
Then for any integer $n\ge 3$ define
\begin{equation}\label{gn}
  g(n):=\begin{cases}
  2m_{n-1}m_{n-2}&\text{if $n>3$ is odd;}\\
  m_{n-1}m_{n-2} &\text{if $n>3$ is even;}\\
  48 &\text{if $n=3$.}
\end{cases}
\end{equation}
For example, $g(4)=6$, $g(5)=20$. For $n>3$, the number $g(n)$ can take the following values: $1$, $2$, $4$, $p$, $2p$, $4p$, where $p$ is an odd prime. 

The numbers $m_i$ and $g(n)$ feature in the following description of the rings
$\varOmega^U$ and $\varOmega^{SU}$ of unitary (complex) and special unitary bordism, respectively.

\begin{theorem}[Milnor--Novikov \cite{novi62}]\label{Ustructure}
The ring $\varOmega^U$ is a polynomial algebra on generators in every even real degree:
\[
  \varOmega^{U}\cong\mathbb{Z}[a_{i}, i\ge 1],\; \deg{a_i}=2i.
\]
The bordism class of a stably complex manifold $M^{2i}$ may be taken to be the $2i$-dimensional generator $a_i$ if and only if
\[
  s_{i}(M^{2i})=\pm m_{i}.
\]
\end{theorem}

\begin{theorem}[Novikov \cite{novi62}]\label{SUstructure}
The ring $\varOmega^{SU}$ with $2$ reversed is a polynomial algebra on generators in every even real degree $>2$:
\[
  \varOmega^{SU}[{\textstyle\frac{1}{2}}]
  \cong\mathbb{Z}[{\textstyle\frac{1}{2}}][y_{i}:\,i\geq 2],  \;  \deg{y_i}=2i.
\]
There exist indecomposable elements $y_i\in\varOmega^{SU}_{2i}$, $i>1$, such that
\[
  s_i(y_i)=g(i+1).
\]
The elements $y_i$ can be taken as polynomial generators of $\varOmega^{SU}[\frac{1}{2}]$.
\end{theorem}

The torsion in $\varOmega^{SU}$ was described by Conner and Floyd in~\cite{co-fl66m}. For the ring structure of $\varOmega^{SU}$, see~\cite{ston68}.

It was shown in~\cite{lu-pa16} that each $y_i$ with $i\ge5$ can be represented by a quasitoric manifold. In the next section we introduce another geometrical representatives for \emph{all} classes $y_{i}$, coming from Calabi--Yau hypersurfaces in toric varieties. A general construction is described below.

\begin{construction}\label{c1dual}
Consider a stably complex manifold $M=M^{2n}$ with the fundamental class $[M^{2n}]$. Let $N=N^{2n-2}$ be a stably complex submanifold dual to the cohomology class $c_1(\mathcal TM)$. That is, we have an inclusion
\[
  i\colon N^{2n-2}\hookrightarrow M^{2n}\quad\text{such that}\quad
  i_{*}([N])=c_{1}(M)\cap [M] \quad\text{in }H_*(M;\Z).
\]
We have $c_1(\mathcal TN)=0$, so $N$ admits an $SU$-structure. 
\end{construction}

A \emph{toric variety} is a normal complex algebraic variety~$V$
containing an algebraic torus $(\C^\times)^n$ as a Zariski open
subset in such a way that the natural action of $(\C^\times)^n$ on
itself extends to an action on~$V$. A nonsingular
complete (compact in the usual topology) toric variety is called a \emph{toric manifold}.
\emph{Projective} toric manifolds $V$ come from convex polytopes $P\subset\R^n$ with vertices in the integer lattice $\Z^n\subset\R^n$. Given such a polytope $P$, we denote by $D_1,\ldots,D_m$ the torus-invariant divisors (codimension-one submanifolds) corresponding to the facets of $P$, and denote by $v_1,\ldots,v_m\in H^2(V;\Z)$ their corresponding cohomology classes.
A projective toric manifold $V$ is a K\"ahler manifold with the total Chern class given by
\[
  c(\mathcal TV)=(1+v_1)\cdots(1+v_m),
\]
so
\[
  c_1(\mathcal TV)=v_1+\cdots+v_m.
\]

The standard complex structure on a toric manifold $V$ is never $SU$ (see \cite[Corollary~4.7]{lu-pa16}), so there are no Calabi--Yau manifolds among $V$. However, the following construction gives Calabi--Yau hypersurfaces in special toric manifolds.

\begin{construction}[Batyrev~\cite{baty94}]
A toric manifold $V$ is \emph{Fano} if its anticanonical class $D_1+\cdots+D_m$ (representing $c_1(V)$) is very ample. In geometric terms, the projective embedding $V\hookrightarrow \C P^s$ corresponding to $D_1+\cdots+D_m$ comes from a lattice polytope $P$ in which the lattice distance from $0$ to each hyperplane containing a facet is~$1$. Such a lattice polytope $P$ is called \emph{reflexive}; its polar polytope $P^*$ is also a lattice polytope.  

The submanifold $N$ dual to $c_1(V)$ (see Construction~\ref{c1dual}) is given by the hyperplane section of the embedding $V\hookrightarrow \C P^s$ defined by $D_1+\cdots+D_m$. Therefore, $N\subset V$ is a smooth algebraic hypersurface in~$V$, so $N$ is a Calabi--Yau manifold of complex dimension $n-1$.

In this way, any toric Fano manifold $V$ of dimension $n$ (or equivalently, any non-singular  reflexive  $n$-dimensional polytope~$P$) gives rise to a canonical $(n-1)$-dimensional Calabi--Yau manifold~$N_P$. Batyrev~\cite{baty94} also extended this construction to singular toric Fano varieties, by considering a special resolution of singularities. This led to defining a family of \emph{mirror-dual} pairs of Calabi--Yau manifolds. 
\end{construction}

The $s$-number of the Calabi--Yau manifold $N_P$ is given as follows.

\begin{lemma}\label{sN}
We have
\[
  s_{n-1}(N)=\bigl\langle(v_1^{n-1}+\cdots+v_m^{n-1})(v_1+\cdots+v_m)-
  (v_1+\cdots+v_m)^n,[V]\bigr\rangle.
\]
\end{lemma}
\begin{proof}
We have an isomorphism of complex bundles $\mathcal TN\oplus\nu\cong i^*\mathcal T V$, where $\nu$ is the normal bundle of the embedding $i\colon N\hookrightarrow V$. Hence, $s_{n-1}(\mathcal TN)+s_{n-1}(\nu)=i^*s_{n-1}(\mathcal T V)$. Now we calculate
\begin{multline*}
  \bigl\langle s_{n-1}(\mathcal T N),[N]\bigr\rangle=
  \bigl\langle -s_{n-1}(\nu)+i^* s_{n-1}(\mathcal T V),[N]\bigr\rangle\\=
  \bigl\langle \bigl(-c_1^{n-1}(\mathcal TV)+s_{n-1}(\mathcal T V)\bigr)c_1(\mathcal TV),     
  [V]\bigr\rangle \\=
  \bigl\langle \bigl(s_{n-1}(\mathcal T V)c_1(\mathcal TV)-c_1^n(\mathcal TV)\bigr),     
  [V]\bigr\rangle.\qedhere
\end{multline*}

\end{proof}

\section{Calabi--Yau generators for the $SU$-bordism ring}
Let $\sigma=(\sigma_1,\ldots,\sigma_k)$ be an unordered partition of $n$ into a sum of $k$ positive integers, that is, $\sigma_1+\cdots+\sigma_k=n$. Let $\varDelta^{\sigma_i}$ be the standard reflexive simplex of dimension~$\sigma_i$. Then $P_\sigma=\varDelta^{\sigma_1}\times\cdots\times\varDelta^{\sigma_k}$ is a reflexive polytope with the corresponding toric Fano manifold $V_\sigma=\C P^{\sigma_1}\times\cdots\times\C P^{\sigma_k}$. We denote by $N_\sigma$ the canonical Calabi--Yau hypersurface in~$V_\sigma$.

We denote by $\widehat P(n)$ the set of all partitions $\sigma$ with parts of size at most $n-2$. That is,
\[
  \widehat P(n):=\{\sigma=(\sigma_1,\ldots,\sigma_k)\colon
  \sigma_1+\cdots+\sigma_k=n,\quad\sigma\ne(n),(1,n-1).\} 
\]
For each $\sigma$ we have the multinomial coefficient $\binom n\sigma=\frac{n!}{\sigma_1!\cdots\sigma_k!}$ and define
\begin{equation}\label{alphasigma}
  \alpha(\sigma):=\binom{n}{\sigma} 
  (\sigma_{1}+1)^{\sigma_{1}}\cdots(\sigma_k+1)^{\sigma_k}.
\end{equation}

\begin{lemma}\label{snumber}
For any $\sigma\in \widehat P(n)$ we have
\[
  s_{n-1}(N_\sigma)=-\alpha(\sigma).
\]
\end{lemma}
\begin{proof}
The cohomology ring of $V_\sigma=\C P^{\sigma_1}\times\cdots\times\C P^{\sigma_k}$ is given by
\[
  H^*(V_\sigma;\Z)\cong\Z[u_1,\ldots,u_k]/(u_1^{\sigma_1+1},\ldots,u_k^{\sigma_k+1}),
\]
where $u_1:=v_1=\cdots=v_{\sigma_1+1}$, $u_2:=v_{\sigma_1+2}=\cdots=v_{\sigma_1+\sigma_2+2}$, $\ldots,$ $u_k:=v_{\sigma_1+\cdots+\sigma_{k-1}+k}=\cdots=v_{\sigma_1+\cdots+\sigma_k+k}=v_m$.
As $\sigma\in \widehat P(n)$, we have $v_i^{n-1}=0$ in $H^*(V_\sigma;\Z)$ for any~$i$. The formula from Lemma~\ref{sN} then gives
\[
  s_{n-1}(N_\sigma)=-\langle(v_1+\cdots+v_m)^n,[V_\sigma]\rangle=
  -\langle((\sigma_1+1)u_1+\cdots+(\sigma_k+1)u_k)^n,[V_\sigma]\rangle.
\]
Evaluating at $[V_\sigma]$ gives the coefficient of $u_1^{\sigma_1}\cdots u_k^{\sigma_k}$ 
in the polynomial above, whence the result follows.
\end{proof}

Mosley~\cite[Proposition~A.2.1]{mosl16} proved that 
\[
  \gcd\limits_{\sigma\in\widehat{P}(n)}\binom n\sigma=m_{n-1}m_{n-2}.
\]
We prove an analogue of this result which takes account of the extra factors in~\eqref{alphasigma}, and therefore obtain a divisibility condition for the numbers $s_{n-1}(N_\sigma)$.

For a prime $p\le n$, write the $p$-adic expansion of $n$ as 
\[
  n=a_{s}p^{s}+\cdots+a_1p+a_{0}. 
\]
Following Mosley, we introduce three particular partions of~$n$.
First, define
\[
  \sigma(p)=\{p^{s},\ldots,p^{s},\ldots,p,\ldots,p,1,\ldots,1\},
\] 
where the number of entries $p^{i}$ is equal to $a_{i}$ for $i\ge 0$. Note that $\sigma(p)\in\widehat P(n)$ whenever $n\ne p^s$ and $n\ne q^r+1$ for any prime $p$ and~$q$. 
For $n=p^s$ define
\[
  \tau(p)=\{p^{s-1},\ldots,p^{s-1}\},
\]
where the number of entries $p^{s-1}$ is $p$. Finally, for $n=q^r+1$ define
\[
  \omega(q)=\{q^{r-1},\ldots,q^{r-1},1\},
\]
where the number of entries $q^{r-1}$ is $q$. Note that $\omega(q)\in\widehat P(n)$ and  $\tau(p)\in\widehat P(n)$ for $n\ge3$.

%Note that the 3 partitions above all lie in $\widehat{P}(n)$ and we can define the following three numbers: 
%$$
%a_{p}(n)=\alpha(\Sigma_{a}(p,n)),\quad b_{p}(n)=\alpha(\Sigma_{b}(p,n)),\quad c_{p}(n)=\alpha(\Sigma_{c}(p,n)).
%$$

For an integer $a$ and a prime $p$, denote by $\ord_{p}a$ the maximal power of $p$ which divides~$a$.

\begin{proposition}[\cite{mosl16}]\label{power}
Let $n\ge3$ be an integer.
\begin{itemize}
\item[(a)]
Let $p$ be a prime such that $n\ne p^s$ and $n\ne p^r+1$.
Then the multinomial coefficient $\binom n{\sigma(p)}$ is not divisible by~$p$.

\item[(b)] Suppose $n=p^s$ for a prime $p$. Then $\ord_p\binom{n}{\sigma}\ge 1$ for any $\sigma\in\widehat{P}(n)$, and $\ord_p\binom n{\tau(p)}=1$;

\item[(c)] Suppose $n=q^r+1$ for a prime $q$. Then $\ord_q\binom{n}{\sigma}\ge 1$ for any $\sigma\in\widehat{P}(n)$, and $\ord_q\binom n{\omega(q)}=1$.
\end{itemize}
\end{proposition}

\begin{lemma}\label{gcd}
For $n\ge3$, we have
\[
  \gcd_{\sigma\in\widehat{P}(n)}\alpha(\sigma)=g(n),
\]
where the numbers $g(n)$ and $\alpha(\sigma)$ are given by~\eqref{gn} and~\eqref{alphasigma} respectively.
\end{lemma}
\begin{proof}
We refer to $\gcd_{\sigma\in\widehat{P}(n)}\alpha(\sigma)$ simply as $\gcd$ throught this proof. For $n=3$ we have $\widehat{P}(n)=\{(1,1,1)\}$ and thus $\gcd=3!\cdot2^{3}=48=g(3)$. In what follows we assume that $n>3$. Consider the following $7$ cases.

\smallskip

\noindent\textbf{I.} $n\ne p^s$, $q^{r}+1$ for any prime $p$ and $q$.

Then $g(n)=1$ if $n$ is even and $g(n)=2$ if $n$ is odd. Take any prime $t$.
First assume $t$ is odd. Then $\ord_t\alpha(\sigma(t))=\ord_t\binom n{\sigma(t)}=0$ by Proposition~\ref{power}~(a), so the gcd can be only a power of~$2$. Now assume $t=2$. Then $\ord_2\alpha(\sigma(2))=\ord_2\binom n{\sigma(2)}+a_0=a_0$ by Proposition~\ref{power}~(a). 

If $n$ is even then $a_0=0$, hence $\gcd=1=g(n)$.  

If $n$ is odd then $a_0=1$ and $\ord_2\alpha(\sigma(2))=1\le\ord_2\alpha(\sigma)$ for any $\sigma\in\widehat P(n)$. The latter inequality holds because in any partition $\sigma=(\sigma_1,\ldots,\sigma_k)$ of  $n=\sigma_1+\cdots+\sigma_k$ at least one number $\sigma_i$ is odd, so $(\sigma_i+1)^{\sigma_i}$ is even and $\alpha(\sigma)$ is divisible by~$2$. It follows that $\gcd=2=g(n)$ in this case.

\smallskip

\noindent\textbf{II}. $n=p^{s}=q^{r}+1$ for some prime $p$, $q$, and $n$ is even. 

Then $p=2$ and $q$ is an odd prime. By Proposition~\ref{power} (b) and~(c), $\alpha(\sigma)$ is divisible by both $p$ and $q$ for any $\sigma\in\widehat{P}(n)$, and 
\[ 
  \ord_2\alpha(\tau(2))=\ord_2{\textstyle\binom n{\tau(2)}}=1,\quad
  \ord_q\alpha(\omega(q))=\ord_q{\textstyle\binom n{\omega(q)}}=1.
\]
For a prime $t$ different from $2$ and $q$ we have $\ord_t\alpha(\sigma(t))=\ord_t\binom n{\sigma(t)}=0$.
Therefore, $\gcd=2q=pq=g(n)$ in this case.

\smallskip

\noindent \textbf{III.} $n=p^{s}=q^{r}+1$ for some prime $p$, $q$, and $n$ is odd. 

Then $p$ is an odd prime and $q=2$. Similarly, by Proposition~\ref{power} (b) and~(c), 
\[ 
  \ord_p\alpha(\tau(p))=\ord_p{\textstyle\binom n{\tau(p)}}=1,\quad
  \ord_2\alpha(\omega(2))=\ord_2{\textstyle\binom n{\omega(2)}}+1=2.
\]
On the other hand, in any partition $\sigma$ of the odd number $n$ at least one number $\sigma_i$ is odd, so $(\sigma_i+1)^{\sigma_i}$ is divisible by~$2$ and $\alpha(\sigma)$ is divisible by $2^2p$.  
Therefore, we get $\gcd=2^{2}p=2pq=g(n)$ in this case. 

\smallskip

\noindent \textbf{IV.} $n=p^s$ for a prime $p$, $n\neq q^{r}+1$ for any prime $q$, and $n$ is even. 

Then $p=2$. By Proposition~\ref{power} (a), $\ord_{q}\alpha(\sigma(q))=0$ for any odd prime~$q$, and $\ord_{2}\alpha(\tau(2))=\ord_2\binom n{\tau(2)}=1\le\ord_2\alpha(\sigma)$ for any $\sigma\in\widehat{P}(n)$. Hence, $\gcd=2=p=g(n)$ in this case.

\smallskip

\noindent\textbf{V.} $n=p^s$ for a prime $p$, $n\neq q^{r}+1$ for any prime $q$, and $n$ is odd. 

Then $p$ is an odd prime. For any prime $t\ne 2,p$, Proposition~\ref{power}~(a) gives $\ord_t\alpha(\sigma(t))=0$. For the prime $2$ we have $\ord_{2}\alpha(\sigma(2))=1\le \ord_2\alpha(\sigma)$  for any $\sigma\in\widehat{P}(n)$, where the equality holds because $n$ is odd. Hence, $\ord_{2}\gcd=1$. Finally, for the odd prime $p$ we have $\ord_p\alpha(\tau(p))=1\le\ord_p\alpha(\sigma)$ for any  $\sigma\in\widehat{P}(n)$. Therefore, $\gcd=2p=g(n)$ in this case.

\smallskip

\noindent\textbf{VI.} $n=q^{r}+1$ for a prime $q$, $n\neq p^s$ for any prime $p$, and $n$ is even. 

Then $q$ is an odd prime. By Proposition~\ref{power}~(c), $\ord_{q}\alpha(\omega(q))=1\le\ord_q\alpha(\sigma)$ for any $\sigma\in\widehat P(n)$, so $\ord_q\gcd=1$. For any prime $t\neq 2,q$, Proposition~\ref{power}~(a) gives $\ord_t\alpha(\sigma(t))=0$, so $\ord_t\gcd=0$. Finally, the $2$-adic expansion of $n$ has $a_{0}=0$, so applying Proposition~\ref{power}~(a) again we get $\ord_{2}\alpha(\sigma(2))=0$. Therefore, $\gcd=q=g(n)$ in this case.

\smallskip

\noindent\textbf{VII.} $n=q^{r}+1$ for a prime $q$, $n\neq p^s$ for any prime $p$, and $n$ is odd. 

Then $q=2$.  

Proposition~\ref{power}~(c) gives
$\ord_2\alpha(\omega(2))=\ord_2\binom n{\omega(2)}+1=2$. On the other hand, Proposition~\ref{power}~(a) implies that $\ord_t\alpha(\sigma(t))=0$ for any prime $t\neq 2$. Also, Proposition~\ref{power}~(c) gives that $\binom{n}{\sigma}$ is divisible by $2$ for any $\sigma\in\widehat{P}(n)$. It follows that $\ord_2\alpha(\sigma)\ge2$, because $n$ is odd. Therefore, $\gcd=2^{2}=2q=g(n)$ in this final case. 
\end{proof}

Now we can state the main result.

\begin{theorem}\label{main}
The $SU$-bordism classes of the canonical Calabi--Yau hypersurfaces $N_\sigma$ in $\C P^{\sigma_1}\times\cdots\times\C P^{\sigma_k}$ 
%the toric Fano varieties over products of simplices %$\varDelta^{\sigma_{1}}\times\ldots\times\varDelta^{\sigma_{k}}$ 
with $\sigma\in\widehat{P}(n)$, $n\ge 3$, multiplicatively generate the $SU$-bordism ring $\varOmega^{SU}[\frac{1}{2}]$.
\end{theorem} 
\begin{proof}
For any $n\ge3$ we use Lemma~\ref{gcd} and Lemma~\ref{snumber} to find a linear combination of the bordism classes $[N_\sigma]\in\varOmega^{SU}_{2n-2}$ whose $s$-number is precisely~$g(n)$. This linear combination represents the polynomial generator $y_{n-1}$ of $\varOmega^{SU}[\frac{1}{2}]$, as described in Theorem~\ref{SUstructure}.
\end{proof}

We actually prove an \emph{integral} result: each element $y_i\in\varOmega^{SU}$ can be represented by an integral linear combination of the bordism classes of Calabi--Yau manifolds~$N_\sigma$. The element $y_i$ is part of a basis of the abelian group $\varOmega^{SU}_{2i}$. There is the following related question:

\begin{question}
Which bordism classes in $\varOmega^{SU}$ can be represented by Calabi--Yau manifolds?
\end{question}

This question is an $SU$-analogue of the following well-known problem of Hirzebruch: which bordism classes in $\varOmega^U$ contain connected (i.\,e., irreducible) non-singular algebraic varieties? If one drops the connectedness assumption, then any $U$-bordism class of positive dimension can be represented by an algebraic variety. Since a product and a positive integral linear combination of algebraic classes is an algebraic class (possibly, disconnected), one only needs to find in each dimension $i$ algebraic varieties $M$ and $N$ with $s_i(M)=m_i$ and $s_i(N)=-m_i$, see Theorem~\ref{Ustructure}. The corresponding argument, originally due to Milnor, is given in~\cite[p.~130]{ston68}. Note that it uses hypersurfaces in $\C P^n$ and a calculation similar to Lemma~\ref{sN}. For $SU$-bordism, the situation is different: if a class $a\in\varOmega^{SU}$ can be represented by a Calabi--Yau manifold, then $-a$ does not necessarily have this property. Therefore, the next step towards the answering the question above is whether $y_i$ and $-y_i$ can be simultaneously represented by Calabi--Yau manifolds. We elaborate on this in the next section.

\section{Low dimensional generators in the $SU$-bordism ring}

Here we describe geometric Calabi--Yau representatives for the generators of the $SU$-bordism ring in complex dimension $\le 4$. Note that for $i\ge 5$, each generator $y_i\in\varOmega^{SU}_{2i}$ can be represented by a quasitoric manifold, by the result of~\cite{lu-pa16}. On the other hand, every quasitoric $SU$-manifold of real dimension $\le8$ is null-bordant by~\cite[Theorem~6.13]{b-p-r10}.

Theorem~\ref{SUstructure} gives the following values of the $s$-number of the elements $y_i\in\varOmega^{SU}_{2i}$ for $i=2,3,4$:
\[
  s_2(y_2)=48, \quad s_3(y_3)=m_3m_2=6,\quad s_4(y_4)=2m_4m_3=20. 
\]
In fact, we have 
\[
  \varOmega^{SU}_4=\Z\langle y_2\rangle,\quad
  \varOmega^{SU}_6=\Z\langle y_3\rangle,\quad
  \varOmega^{SU}_8=\Z\langle {\textstyle\frac14}y^2_2,y_4\rangle,
\]
(see~\cite[p.~266]{ston68}). Note that $8$ is the first dimension where the difference between $\varOmega^{SU}$ and $\varOmega^{SU}[\frac12]$ becomes relevant (apart from the torsion elements), as the square of the $4$-dimensional generator $y_2$ is divisible by~$4$. Note also that $y_2$, $y_3$ and $y_4$ are integral basis elements in the corresponding dimensions.

\begin{example}
Consider the Calabi--Yau hypersurface $N_{(3)}\subset\C P^3$ corresponding to the partition $\sigma=(3)$.
We have $c_1(\mathcal T\C P^3)=4u$, where $u\in H^2(\C P^3;\Z)$ is the canonical generator dual to a hyperplane section. Therefore, $N_{(3)}$ can be given by a generic quartic equation in homogeneous coordinates on $\C P^3$. The standard example is the quartic given by $z_0^4+z_1^4+z_2^4+z_3^4=0$, which is a $K3$-surface. Lemma~\ref{sN} gives
\[
  s_3(N_{(3)})=\langle4u^2\cdot4u-(4u)^3,[\C P^3]\rangle=-48,
\]
so $N_{(3)}$ represents the generator $-y_2\in\varOmega^{SU}_4$.

Note that Theorem~\ref{main} gives another representative for the same generator~$-y_2$. Namely, the only partition of $n=3$ which belongs to $\widehat P(n)$ is $(1,1,1)$. The corresponding Calabi--Yau surface is $N_{(1,1,1)}\subset\C P^1\times\C P^1\times\C P^1$. We have 
\[
  c_1(\C P^1\times\C P^1\times\C P^1)=2u_1+2u_2+2u_3,
\]
so $N_{(1,1,1)}$ is a surface of multidegree $(2,2,2)$ in $\C P^1\times\C P^1\times\C P^1$. Lemma~\ref{snumber} gives $s_3(N_{(1,1,1)})=-\alpha(1,1,1)=-48$, so $N_{(1,1,1)}$ also represents~$-y_2$.

On the other hand, the additive generator $y_2\in\varOmega_4^{SU}$ cannot be represented by a compact complex surface. This is proved in~\cite[Theorem~3.2.5]{mosl16} by analysing the classification results on complex surfaces. It is easy to see that a complex surface $S$ with $H^1(S;\Z)=0$ (which holds for all canonical Calabi--Yau hypersurfaces in toric Fano varieties) cannot represent~$y_2$. Indeed, such $S$ has the Euler characteristic $c_2(S)=\chi(S)\ge2$, while $s_2(y_2)=48=-2c_2(y_2)$, so $c_2(y_2)=-24$ is negative.
\end{example}

The situation is different in complex dimension~$3$, where we have

\begin{proposition}\label{CY3}
In complex dimension $3$ both additive generators $y_3$ and $-y_3$ of $\varOmega_6^{SU}$ are representable by canonical Calabi--Yau hypersurfaces in toric Fano varieties.
Namely, such a Calabi--Yau hypersurface $N$ represents $\pm y_3$ if and only if its Hodge numbers satisfy
\[
  h^{1,1}(N)-h^{2,1}(N)=\pm1.
\]
For each number $h^{1,1}$ within the range
\[
  16\le h^{1,1}\le 90,
\]
there exists a Calabi--Yau hypersurface $N$ representing $y_3$ with the second Betti number $b^2(N)=h^{1,1}$. Similarly, for each number $h^{1,1}$ within the range 
\[
  15\le h^{1,1}\le 89,
\]
there exists a Calabi--Yau hypersurface $N$ representing $-y_3$ with the second Betti number $b^2(N)=h^{1,1}$.
Furthermore, Calabi--Yau hypersurfaces representing $y_3$ and $-y_3$ can be chosen to be mirror dual in the sense of~\cite{baty94}.
\end{proposition}

\begin{proof}
For the generator $y_3\in\varOmega_6^{SU}$ we have $6=s_3(y_3)=3c_3(y_3)$, so a complex $SU$-manifold $N$ representing $y_3$ must have the Euler characteristic $\chi(N)=c_3(N)=2$. On the other hand, if $N$ is a K\"ahler $n$-manifold, we have $\chi(N)=\sum_{i=0}^{2n}(-1)^ib^i=\sum_{p,q=0}^n(-1)^{p+q}h^{p,q}$, where $b^i=\sum_{p+q=i}h^{p,q}$ is the $i$th Betti number of $N$, and the Hodge numbers satisfy $h^{p,q}=h^{q,p}$ and $h^{p,q}=h^{n-p,n-q}$. Furthermore, for Calabi--Yau hypersurfaces $N$ in toric Fano varieties we additionally have $h^{0,0}=h^{n,0}=1$ and $h^{i,0}=0$ for $0<i<n$, see~\cite[Theorem~4.1.9]{baty94}. Therefore, in our case $n=3$ we have 
\[
  b^1=2h^{1,0}=0,\quad b^2=2h^{2,0}+h^{1,1}=h^{1,1},
  \quad b^3=2h^{3,0}+2h^{2,1}=2+2h^{2,1},
\]
and
\[
  \chi(N)=2b^0-2b^1+2b^2-b^3=2(h^{1,1}-h^{2,1}).
\]
It follows that $N$ represents $y_3$ if and only if $h^{1,1}-h^{2,1}=1$, and similarly for~$-y_3$. 

The fact that such $N$ exist follows by analysing the database~\cite{kr-sk} of reflexive polytopes and the Calabi--Yau hypersurfaces in their corresponding toric Fano varieties. This database contains the full list of 473,800,776 reflexive polytopes in dimension $4$, and the list of Hodge numbers of the corresponding Calabi--Yau $3$-folds. From there one can see that for each $h^{1,1}$ satisfying $16\le h^{1,1}\le 90$ there exists a reflexive $4$-polytope with the corresponding Calabi--Yau $3$-fold satisfying $h^{1,1}-h^{2,1}=1$, and if $h^{1,1}$ is not within this range, then there is no Calabi--Yau $3$-fold with $h^{1,1}-h^{2,1}=1$ coming from a toric Fano variety. In the case of the identity $h^{1,1}-h^{2,1}=-1$, the possible range is $15\le h^{1,1}\le 89$. The statement about the mirror duals follows by observation that if $N$ and $N^*$ are mirror dual Calabi--Yau $3$-folds, then $h^{1,1}(N)=h^{2,1}(N^*)$ and $h^{2,1}(N)=h^{1,1}(N^*)$, see~\cite[Corollary~4.5.1]{baty94}.
\end{proof}

\begin{example}
Theorem~\ref{main} gives the following representatives for the generators $y_3\in\varOmega_6^{SU}$ and $y_4\in\varOmega_8^{SU}$:
\[
  y_3=15 N_{(2,2)}-19 N_{(1,1,1,1)},\quad y_4=56 N_{(1,1,3)}-59 N_{(1,2,2)}.
\]
As we have seen from Proposition~\ref{CY3}, the generator $y_3$, as well as~$-y_3$, can be represented by a single Calabi--Yau $3$-fold $N$ corresponding to a certain reflexive polytope.
The case of $y_4$ is will be addressed in a subsequent work.
\end{example}

One interesting aspect of low-dimensional geometric representatives $M$ for the bordism rings $\varOmega^U$ and $\varOmega^{SU}$ is that the rigidity property of Hirzebruch genera for manifolds $M$ with torus action lead to remarkable functional equations. See~\cite[\S9.7]{bu-pa15} and~\cite{buch18} for the details on this subject. The case of toric Fano manifolds corresponding to $2$-dimensinal reflexive polytopes is treated in~\cite{buch18}; it would be interesting to extend this treatment to the $3$- and $4$-dimensional case.

The authors are grateful to Victor Buchstaber for useful recommendations and stimulating discussions of the results of this work.

\end{document}